\documentclass{amsart}
\usepackage{amssymb ,amsmath, graphicx, enumitem, color}
\usepackage{caption}
\usepackage{subcaption}
\usepackage[all]{xy}
\usepackage{chngcntr}

\newtheorem{thm}{Theorem}[section]
\newtheorem{prop}[thm]{Proposition}
\newtheorem{lem}[thm]{Lemma}
\newtheorem{corollary}[thm]{Corollary}

\theoremstyle{definition}
\newtheorem{defn}[thm]{Definition}

\theoremstyle{remark}
\newtheorem*{rem}{Remark}
\newtheorem*{notation}{Notation}
\newtheorem*{ques}{Question}

\renewcommand{\H}{\ensuremath{\mathcal{H}} }

\renewcommand{\u}{\ensuremath{\mathbf{u}} }
\renewcommand{\v}{\ensuremath{\mathbf{v}} }
\newcommand{\w}{\ensuremath{\mathbf{w}} }
\newcommand{\W}{\ensuremath{\mathbf{W}} }
\newcommand{\V}{\ensuremath{\mathbf{V}} }
\renewcommand{\a}{\ensuremath{\mathbf{s}} }
\newcommand{\axisc}{\ensuremath{\mathbf{c}} }
\newcommand{\E}{\ensuremath{\mathbb{E}} }

\newcommand{\dis}{\ensuremath{\parallel}}
\newcommand{\inter}{\ensuremath{\perp}}

\thanks {Mike Carr was partially supported by NSF grant DMS-1106726}

\begin{document}

\begin{abstract}
We show that non-abelian two-generator subgroups of right-angled Artin groups are quasi-isometrically embedded free groups.  This provides an alternate proof of a theorem of A. Baudisch: that all two-generator subgroups are free or free abelian. Additionally, it shows that they are quasi-isometrically embedded.  Our theorem also gives a method for detecting groups that are not isomorphic to a subgroup of any RAAG.  We present some counterexamples in subgroups with more than two generators.
\end{abstract}
\title[Two-generator subgroups of RAAGs are QI embedded]{Two-generator subgroups of right-angled Artin groups are quasi-isometrically embedded}
\author{Mike Carr}
\maketitle

\section{Introduction}

\begin{defn}
Given a combinatorial graph $\Gamma$ with vertex set $V$ and edge set $E$, the \emph{right-angled Artin group} $A_\Gamma$ is the group presented by the generating set $V$ and the relations $\{ v_i v_j = v_jv_i ~|~ (v_i,v_j) \in E \}$.
\end{defn}

Right-angled Artin groups (RAAGs) constitute a spectrum between free abelian groups, given by complete graphs, and free groups, given by edgeless graphs.  The simple definition belies that fact that complicated groups can exist as subgroups of right-angled Artin groups. Crisp and Weist showed that nearly all surface groups can be embedded in some RAAG \cite{CW04}.   In addition to providing a wealth of interesting examples, RAAG subgroups have underpinned several recent results in group theory and topology.  The recent proof of the virtual Haken conjecture by Ian Agol \cite{AGM12} relied on showing that every hyperbolic 3-manifold group had a finite index subgroup that embeds in a RAAG via a map defined by Haglund and Wise \cite{HW08}.  Hsu and Wise also used embeddings of subgroups to show that certain graph groups are linear \cite{HW10}.

The behavior of two-generator subgroups of RAAGs is much more circumscribed.  A theorem of A. Baudisch completely describes their group structure.

\begin{thm}\cite[1.3]{Bau81}
\label{2gen}
Every two generator subgroup of $A_\Gamma$ is either free or free abelian.
\end{thm}

This description of two-generator subgroups is interesting for at least two reasons.  First, it passes to subgroups.  If this description holds for a group $G$, then it holds for every subgroup $H<G$.  Second, two generator subgroups of many other well-studied groups have this description or some approximation of it.  In mapping class groups, it is false in general, but given two generators, one can take appropriate powers of each such that the group those powers generate is free or free abelian.

Two generator subgroups of pure braid groups do fit this description. This fact that was used by Scrimshaw \cite{Scr10} and Kim and Koberda \cite{KK13} who recovered Baudisch's result by embedding any $A_\Gamma$ in some pure braid group.  

Our main result concerns the metric structure of two-generator subgroups in a right-angled Artin group.  We use the word metric to realize each group as a metric space.  The desired metric property is the following.

\begin{defn}
A \emph{quasi-isometric embedding} is a map $f:X \to Y$ between metric spaces such that there are constants $\lambda \geq 1$ and $\epsilon \geq 0$ that satisfy
$$\frac{1}{\lambda} d_X(p, q) -\epsilon \leq d_Y(f(p), f(q)) \leq \lambda d_X(p,q) + \epsilon$$
for all $p, q \in X$.
\end{defn}

The metric behavior of abelian subgroups of a RAAG is already well understood, via the Flat Torus Theorem (see in \cite[II.7.1 and II.7.17]{bh99}).  Our main theorem resolves the nonabelian case.

\begin{thm}
\label{main}
Let $F= \langle u, v \rangle$ be a free group.  If $\phi: F \to A_\Gamma$ is a homomorphism, such that $\phi(u)$ and $\phi(v)$ do not commute, then $\phi$ is a quasi-isometric embedding.
\end{thm}

We retrieve Baudisch's original, algebraic description (\ref{2gen}) as a corollary.   Combining with the abelian case, we produce the following metric description.

\begin{corollary}
\label{2genqie}
Every two generator subgroup of $A_\Gamma$ is a quasi-isometrically embedded free or free abelian group.
\end{corollary}

One consequence is that if an injective homomorphism $\psi:F_2 \to G$ is not quasi-isometric, then $G$ is not isomorphic to a subgroup in any RAAG.  

There is also a divergence between the algebraic and metric treatments in higher rank free groups.  We produce a free subgroup of a RAAG whose inclusion is not a quasi-isometric embedding.

\section{Background}
This section will recall some standard results from geometric group theory, beginning with a discussion of quasi-isometric embeddings.

\subsection{Quasi-isometries}

It is a straightforward exercise to show that compositions of quasi-isometric embeddings are quasi-isometric embeddings.  The next natural step is to define a relation based on those embeddings that have an inverse, defined as follows.

\begin{defn}
A \emph{quasi-isometry} is a quasi-isometric embedding such that all of $Y$ lies within a fixed distance of the image $f(X)$.  If such a map exists, then $X$ and $Y$ are \emph{quasi-isometric}.
\end{defn}

One can show that a quasi-isometry $f$ has a quasi-inverse, a quasi-isometry $g$ such that $f\circ g$ and $g\circ f$ move points a bounded distance.  ``Quasi-isometric" forms an equivalence relation on metric spaces.  It is the natural equivalence relation for the study of groups as metric spaces.  From any generating set $S$ of a group $G$, one can produce the word metric $d$ where $d(g_1, g_2)$ is the length of the shortest word in $S$ that represents $g_1^{-1}g_2$.  We can't directly use geodesics to study $G$, since it has none.  Fortunately the Cayley graph of $G$ does have geodesics, and it is straightforward to prove the following:

\begin{prop}
If $G$ has generating set $S$ then the Cayley graph $Cay(G, S)$ (with unit length edges) is quasi-isometric to the group $G$ given the word metric from $S$.
\end{prop}

The other natural appeal of quasi-isometry is that the word metric depends on the choice of generating set $S$.  Quasi-isometry resolves this ambiguity.  

\begin{prop}
If two finite generating sets $S_1$ and $S_2$ of a group $G$ give word metrics $d_1$ and $d_2$, then the identity map on $G$ is a quasi-isometry between $(G, d_1)$ and $(G, d_2)$.  
\end{prop}

A more general and more powerful theorem concerning quasi-isometries of groups is the following:

\begin{thm}[The \includegraphics{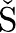}varc-Milnor Lemma \cite{Mil68}]
If a group $G$ acts properly, cocompactly and by isometries on a geodesic metric space $X$, then $G$ is finitely generated.  Furthermore for any $x \in X$, the orbit map $g \mapsto gx$ is a quasi-isometry.   
\end{thm}

An accessible proof is found in \cite[I.8.19]{bh99}.

\subsection{CAT(0) spaces}

We will denote a geodesic segment between points $a$ and $b$ in a geodesic metric space by $[a, b].$

\begin{defn}
A geodesic metric space $X$ has the CAT(0) property if for any three points $p, q, r \in X$, any geodesic triangle $[p, q] \cup [q, r] \cup [r,s ]$  is at least as thin as a comparison triangle $[P,Q] \cup [Q,R] \cup [R,P]$ with equal length sides in the euclidean plane $\E$.   That is, for any $x \in [p,q]$ and $y \in [q, r]$, the corresponding points $X \in [P, Q]$ and $Y \in [Q, R]$ with $d_X(x,p) = d_\E(X, P)$ and $d_X(y,q)=d_\E(Y, Q)$ have the property:  

$$d_X(x, y) \leq d_\E(X, Y).$$

If every point in $X$ has a neighborhood that is CAT(0), then we say $X$ is \emph{non-positively curved}.
\end{defn}

The CAT(0) property has many useful consequences.  A general exposition is found in \cite[II.1]{bh99}. 

\begin{prop}\label{vary}
If $X$ is CAT(0) and $p, q, r \in X$ with $d_X(q, r)< \epsilon$ then the geodesics $[p,q]$ and $[p, r]$ each lie within an $\epsilon$ -neighborhood of the other.
\end{prop}

This follows directly from the equivalent fact in the comparison triangle.  Here is a powerful application.

\begin{corollary}
If $X$ is CAT(0) then there is only one geodesic between any two points $p$ and $q$.  
\end{corollary}

\begin{corollary}
A CAT(0) space is contractible.
\end{corollary}

The contraction moves each point to the basepoint along the unique geodesic between them at constant speed.  Continuity follows directly from Proposition \ref{vary} and the fact that geodesics are continuous maps.

\begin{prop}\label{local}
In a CAT(0) space, local geodesics are (global) geodesics, and are thus unique.
\end{prop}

Non-positive curvature can also be used to study covering spaces and fundamental groups.  One critical ingredient is the following generalization of a theorem from Riemannian geometry, whose name it also shares.

\begin{thm}[The Cartan-Hadamard Theorem for Non-Positively Curved Spaces] \cite[Theorem 1]{AB90}
If $X$ is non-positively curved, then the universal cover $\widetilde{X}$ is CAT(0).
\end{thm}

The following standard application illustrates the connection between metric and group theoretic aspects of maps.  A proof is in \cite[II.4.14]{bh99}.

\begin{prop}
\label{isometric embedding}
If $X$ is non-positively curved, $Y$ is a geodesic metric space and $f: Y \to X$ is a locally isometric embedding, then the induced map $f_* :\pi_1(Y) \to \pi_1(X)$ is injective.
\end{prop}

Furthermore, by the the \includegraphics{vs.pdf}varc-Milnor Lemma, the map $f_*$ is a quasi-isometric embedding, since it is the composition of two quasi-isometries and an isometric embedding:
\begin{align*}
\pi_1(Y) \to \widetilde{Y} \to \widetilde{X} \to \pi_1(X)
\end{align*}

\subsection{Isometries of CAT(0) spaces}

Like in the hyperbolic plane, isometries of a CAT(0) space are classified as either elliptic, parabolic, or hyperbolic.  

If $G$ is a group acting properly by isometries on $X$ such that $X/G$ is compact, then a standard result \cite[II.6.10.2]{bh99} is that all $g \in G$ are either elliptic or hyperbolic.  If $G$ is a deck group, then it acts freely and all its isometries are hyperbolic.

A hyperbolic isometry of the hyperbolic plane has an axis, a bi-infinite geodesic preserved by the isometry. 

If $g$ is a hyperbolic isometry of a CAT(0) space, then it also has an axis.  This axis may not be unique, but any two axes lie in bounded neighborhoods of each other.  This is clear since $g$ preserves distance.  A CAT(0) argument shows that any two axes are actually parallel (see \cite[II.6.8]{bh99}).

\subsection{Cube complexes}

\begin{defn}
A \emph{cube complex} is a metric space $X$ that is the union of metric cubes $[-\frac{1}{2}, \frac{1}{2}]^n$ glued by isometries of their faces.
\end{defn}

Finite dimensional cube complexes are geodesic metric spaces \cite[I.7.19]{bh99}.

Given a RAAG $A_\Gamma$, we take a cube $[-\frac{1}{2}, \frac{1}{2}]^{\#V}$ with the euclidean metric and with opposite faces identified to make an $n$-torus.  The faces of this torus form a complex of metric cubes, specifically, each edge is an interval associated to an element of $V$.  We consider the subcomplex of only those cubes whose edges all commute pairwise in $A_\Gamma$.

This subcomplex has fundamental group $A_\Gamma$.  We call it the \emph{Salvetti complex}, which we will denote $X_\Gamma$.   With the path metric, $X_\Gamma$ satisfies Gromov's link condition \cite[4.2c]{Gro87} for a non-positively curved cube complex \cite[3.1.1]{CD95}.  

In addition, by the \includegraphics{vs.pdf}varc-Milnor Lemma, the universal cover $\widetilde{X_\Gamma}$ is quasi-isometric to $A_\Gamma$.

\subsection{Hyperplanes}

\begin{defn}
Given a metric $n$-cube: $[-\frac{1}{2}, \frac{1}{2}]^{n}$, we define the $i^{th}$ \emph{midcube} to be a set of points that is $0$ in the $i^{th}$ coordinate.  A midcube naturally inherits the metric and cell structure of an $(n-1)$-cube. 

In a cube complex $X$ we define a relation on all midcubes in $X$.  Two midcubes are related if and only if they share a vertex.  There is a unique weakest equivalence relation generated by this relation.  This equates any midcubes that can be connected by a succession of vertices.  The union of cubes in an equivalence class is a \emph{hyperplane}.  Figure \ref{hyperplanefig} shows examples in a small complex.  

We'll say that the vertices of $X$ that belong to edges of $X$ crossing a hyperplane $h$ are \emph{adjacent} to $h$.
\end{defn}

\begin{figure}[ht]
\centering
\def\svgwidth{90mm}
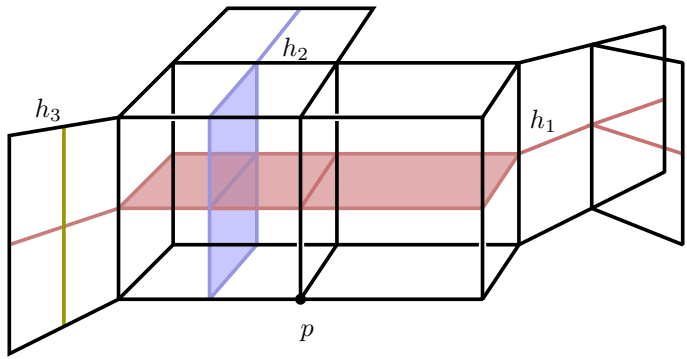
\caption[3 hyperplanes in a cube complex]{3 hyperplanes in a cube complex. $p$ is adjacent to $h_1$ and $h_2$ but not $h_3$.}
\label{hyperplanefig}
\end{figure}

If $X$ is CAT(0) then a hyperplane $h$ contains at most one midcube of a given cube, and divides $X$ into two connected components \cite[4.10]{Sag95}.  The cubes containing $h$ form an $h \times I$ neighborhood with a metric $\sqrt{d_h^2 + d_I^2}$.  This makes geodesics near $h$ very easy to compute, and since geodesic segments are unique in $X$ we obtain the following standard results:

\begin{enumerate}
\item If $p, q \in h$, then the geodesic segment from $p$ to $q$ (in $X$) is also contained in $h$.  Specifically, $h$ is simply connected.
\item A geodesic in $X$ is either contained in $h$, disjoint from $h$ or intersects $h$ transversely at a single point.
\end{enumerate} 

We'll make use of the following lemma in our proof of the main theorem.

\begin{lem}
\label{transverse}
If $L$ is a bi-infinite geodesic in a CAT(0) cube complex that transversely intersects a hyperplane $h$,
then the distances from points on $L$ to $h$ is unbounded, traveling in either direction along $L$.
\end{lem}

\begin{proof}
Suppose $L\cap h =p$.  Take $q \in L$ at distance $1$ from $p$.  Now take $r \in L$ distance $t>1$ from $p$ along the same direction as $q$. Let $r'$ be a point on $h$ of minimum distance to $r$.    

Take the comparison triangle $(P, R, R')$ and let $Q'$ be the point on  $[P,R']$ with $d(P, Q')=\frac{1}{t} d(P, R'$).  The corresponding point $q'$ lies in $h$ along the geodesic from $p$ to $r'$.  Then we have 
$$d(q, h)  \leq d(q, q') \leq d(Q, Q') = \frac{1}{t} d(R,R') = \frac{1}{t} d(r,h).$$ 

Since $d(q, h)$ is a constant, we can make $r$ arbitrarily far from $h$ by increasing $t$.
\end{proof}

\begin{figure}[ht]
\centering
\def\svgwidth{90mm}
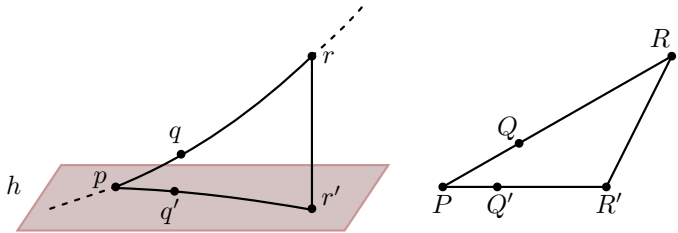
\caption{The comparison triangle for Lemma \ref{transverse}}
\label{transversefigure}
\end{figure}

Here is an immediate consequence.  

\begin{corollary}
If $L_1$ and $L_2$ are axes of $w$, then $L_1$ properly intersects $h$ if and only if $L_2$ does.
\end{corollary}

Furthermore $h$ has the structure of a cube complex, given by the midcubes it comprises.  If $X$ is CAT(0), it is straightforward to check that so is $h$. If $w$ is a hyperbolic isometry of $X$ preserving $h$, then $w$ has an axis in $h$ which must also be an axis in $X$.

\subsection{Special cube complexes}

The theory of special cube complexes was introduced by Haglund and Wise \cite{HW08}.  Special cube complexes are cube complexes whose hyperplanes avoid certain pathologies.  They are defined in such a way that guarantees a locally isometric embedding in the Salvetti complex of some RAAG. Thus if $C$ is a special cube complex, then there is an injective homomorphism from $\pi_1(C)$ to some RAAG.

Later, Hsu and Wise used special cube complexes to show that many free by cyclic graph groups virtually embed in RAAGs.  

\begin{thm}\label{hsuwise} \cite{HW10}
Let $G$ be a graph of groups with free vertex groups and cyclic edges groups.  Suppose $G$ contains no Baumslag-Solitar groups, or equivalently is word hyperbolic.  Then $G$ acts on a cube complex $C$, and $G$ contains a finite index subgroup $H$ whose action on $C$ is special.  As a result, $H$ quasi-isometrically embeds in some RAAG.
\end{thm}

This theorem applies to a free by cyclic HNN extension that we'll use as an example later. 

A final application is the following result, due to Ian Agol.  It is part of his proof of the virtual Haken conjecture.

\begin{thm}\cite{AGM12}
Let $C$ be a non-positively curved cube complex such that $\pi_1(C)$ is word hyperbolic.  Then $C$ has a finite cover which is special.
\end{thm}

These two theorems are statements of existence.  Neither gives a method for finding the desired cover.  

\section{$X^{(1)}$ geodesics}

\subsection{Interaction with hyperplanes}
In this section, $A$ will denote a RAAG.  We'll use $X$ to refer to the universal cover of its Salvetti complex.

Most of our arguments exploit our ability to relate the combinatorics of $A$ to the geometry of $X$.  We can, in fact, see the group directly in $X$ by tracing our paths not in $X$ but in its $1$-skeleton.  

We'll call paths in $X$ that stay in $X^{(1)}$ and don't change direction mid-edge \emph{edge paths}.  Edge paths between vertices are nothing more than words in the generators of $A$.   The path follows the edge corresponding to each letter in turn.  

Unless $A$ is free, $X^{(1)}$ is not CAT(0) and does not have unique geodesics.  We'll use $[p, q]$ to denote some choice of minimal length path between $p$ and $q$ in $X^{(1)}$.  We'll call this an \emph{$X^{(1)}$ geodesic segment}.  Given an edge path from $p$ to $q$ encoded by $w \in A$, finding an $X^{(1)}$ geodesic $[p,q]$ is equivalent to finding a shortest form for $w$.   The distance from $1$ to $w$ in the word metric, denoted $|w|$ is also the length of $[p,q]$.  Here are some quick results connecting $X^{(1)}$ geodesics segments to hyperplanes.

Observe that an edge path is transverse to the hyperplanes of $X$.  Furthermore, two vertices of $X$ lie on the same side of a hyperplane $h$ if and only if every edge path between them crosses $h$ an even number of times.

\begin{lem}
\label{once}
An $X^{(1)}$-geodesic segment crosses no hyperplane more than once. 
\end{lem}

\begin{proof}
Suppose $[p,q]$ does cross some hyperplane twice.  Consider an innermost pair of crossings.  Let $h$ be the hyperplane they cross.  The path between the crossings does not meet any hyperplane $h'$ disjoint to $h$, otherwise it would have to recross $h'$ in order to return to $h$, violating our choice of an innermost pair of crossings.  Thus the path stays adjacent to $h$, in an $h \times I$ neighborhood.  This means $[p,q]$ can be shortened by removing two crossings of $h$, so it is not a geodesic segment.
\end{proof}

\begin{corollary}
An $X^{(1)}$ geodesic segment is any segment that crosses no hyperplane more than once.  Given two points, the $X$ geodesic segment and $X^{(1)}$ geodesic segments between them all cross the same set of hyperplanes.
\end{corollary}

The following corollary is an application of the triangle inequality.

\begin{corollary}
\label{axis}
Given $w \in A$ not equal to $1$ and a vertex $p \in X$ such that $d(p, wp)$ is minimal, the set $\xi = \displaystyle \bigcup_{i \in \mathbb{Z}} w^i[p,wp]$ is an axis of $w$, that is a $w$-invariant bi-infinite $X^{(1)}$ geodesic.
\end{corollary}

Note that $\xi$ depends both on the choice of $p$ and $[p, wp]$, so it is far from unique.  However, $\xi$ does remain within bounded distance of any (non-$X^{(1)}$) axis $L$ of $w$, since it is the $w$ orbit of a (bounded) segment.

The last lemmas rely on the specific construction of $X$, rather than general CAT(0) hyperplane facts.  By inspection, the hyperplanes of $X/A$ meet a single edge each.  Thus we can assign to each hyperplane in $X$ a \emph{type}, according to which hyperplane in $X/A$ lifts to it.  Moreover, since the hyperplanes of $X/A$ do not self-intersect, distinct hyperplanes of the same type in $X$ do not intersect.  Finally, note that the action of $A$ preserves type, as well as the orientation of the edges that cross each hyperplane.

\begin{lem}[Slope Lemma]
\label{slope}
Suppose $w \in A$ has axis $\xi$.  Suppose $p$ is a point on $\xi$, and $q$ is a point on a hyperplane $h$.  $h$ crosses $\xi$ beyond $wp$ (that is, on the ray $\xi = \displaystyle \bigcup_{i=1}^\infty w^i[p,wp]$) if and only if $h$ crosses $[wp, wq]$. Note figure \ref{slopefig}.
\end{lem}

\begin{figure}[ht]
\centering
\def\svgwidth{90mm}
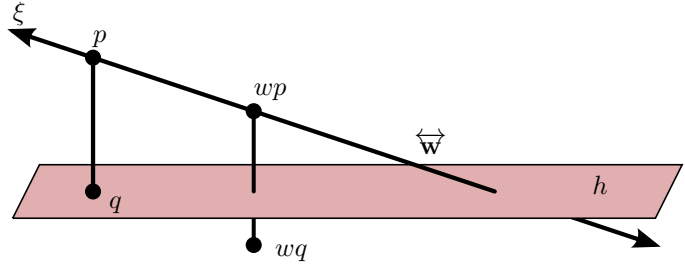
\caption{An illustration of the Slope Lemma}
\label{slopefig}
\end{figure}

\begin{proof}
If $h$ crosses $\xi$ then $wh \neq h$.  If the crossing is beyond $wp$ then $h$ must cross $[wp, wq]$ to avoid crossing $wh$.

On the other hand, if $h$ crosses $[wp, wq]$, we have that $w^{-1}h$ crosses $[p, q]$. Since $wh \neq h$, $w^{-1}h$ must also cross either $\xi$ or $[wp, wq]$.  Continuing by induction, we either get that some $w^{-k}h$ crosses $\xi$ (thus producing the desired conclusion) or that infinitely many hyperplanes cross $[p, q]$, an absurdity.
\end{proof}

If $h_1$ and $h_2$ are two hyperplanes in $X$ whose projections in $X/A$ intersect transversely we write $h_1 \inter h_2$.  If their projections are disjoint or identical we write $h_1 \dis h_2$.  No hyperplane of $X/A$ self intersects, so if $h_1 \dis h_2$, then $h_1$ and $h_2$ are either identical or disjoint.  However not all pairs $h_1 \inter h_2$ intersect.

If two hyperplanes $h_1 \inter h_2$ are adjacent to the same vertex, then they do in fact intersect in a square that meets that vertex.  Also, only one edge of each type and direction meets a given vertex.  These facts are immediate from our construction of $X$.  They are also, however, the remaining two conditions for the action of $A$ on $X$ to be  \emph{special}.  The following lemma is true for any group acting specially on a CAT(0) cube complex.

\begin{lem}
\label{adjacent}
Let $p$ be a vertex and $h$ a hyperplane.  The vertex $p$ is adjacent to $h$ if and only if every geodesic segment $[q, p]$ from any $q \in h$ crosses only hyperplanes $h_i \inter h$.
\end{lem}

\begin{proof}
Let $p_i$ be the $i^{th}$ vertex in $[q,p]$, and $h_i$ be the hyperplane crossing $[p_i, p_{i+1}]$.  Suppose $h_i \inter h$ for all $i$.  We'll show $p$ is adjacent to $h$ by induction.  The first vertex, $p_1$, is adjacent by definition.  Now suppose $p_i$, is adjacent to $h$, via an edge we'll denote $e$. The edges $[p_i, p_{i+1}]$ and $e$ span a square, $h$ is a midcube of the square, and $p_{i+1}$ is still adjacent to $h$.

On the other hand if some $h_i \dis h$ then $h_i$ separates $h$ from $[p_{i+1}, p]$, since $[q,p]$ crosses no hyperplane twice.  Thus any path from $h$ to $p$ must cross $h_i$, and $p$ is not adjacent to $h$.
\end{proof}

Any element of $A$ that preserves $h$ also preserves the set of adjacent vertices on each side.  Given a point $p$ adjacent to $h$, the stabilizer subgroup of $h$ in $A$ sends $p$ along the edges that cross $h_i \inter h$.  In a RAAG, these translations have a particular form.  They correspond to the generators that commute with the generator whose edge crosses $h$.  Thus the stabilizer subgroup is conjugate to a subgroup generated by vertices of $\Gamma$.  This characterization allows us to conclude that if $w^ih=h$ then $wh=h$. which permits a useful variant of the Slope Lemma (\ref{slope}).

\begin{lem}[Parallel Axis Lemma]
\label{parallel}
Let $w \in A$ have axis $\xi$ as above, and $p \in \xi$.  Suppose that some hyperplane $h$ does not meet $\xi$ and $r$ is some point in $X$ such that $[p,r]$ meets $h$.  If $h$ intersects $w^n[p, r]$ for any $n \neq 0$, then $wh=h$.
\end{lem}

\begin{figure}[ht]
\centering
\def\svgwidth{90mm}
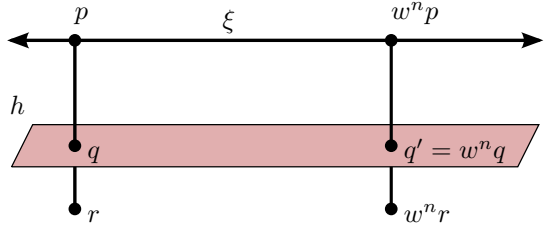
\caption{An illustration of the Parallel Axis Lemma}
\label{parallelfig}
\end{figure}

\begin{proof} 
Let $q = h\cap [p, r]$.  Suppose that for some $n \neq 0$, $h$ meets $w^n[p, r]$ at $q'$.  If $d(p, q) < d(w^ip, q')$, then the Slope Lemma implies that $h$ meets $\xi$ beyond $p$ (on the ray not containing $w^np$. Similarly, if  $d(p, q) > d(w^np, q')$, then $h$ meets $\xi$ beyond $w^np$.  However, $h$ doesn't meet $\xi$ at all.  We conclude that $d(p, q) = d(w^np, q')$ and $w^nq=q'$ (see Figure \ref{parallelfig}). Thus $w^nh=h$.  As noted above, if $w^n$ is in the stabilizer of $h$ then so is $w$.
\end{proof}

\subsection{A standard form for $X^{(1)}$ geodesics}

The Salvetti complex $X/A$ has a single vertex. If we choose a base vertex of $X$, which we'll call $1$, then every other vertex has the form $w\cdot 1$ for some $w \in A$.  For brevity, we'll let $w$ denote the vertex $w \cdot 1$ as well as the group element.  

Given an element $w \in A$ we'll choose and label a few useful objects, also illustrated in Figure \ref{standardfig}:

\begin{enumerate}
\item A vertex $a_w$ on an axis of $w$ that is minimal distance from $1$ among all vertices that lie on axes of $w$
\item A geodesic $\w$ from $a$ to $wa$
\item A geodesic $\a_w$ from $1$ to $a$.
\item The concatenation $\displaystyle \overleftrightarrow{\w} = \bigcup_{i \in \mathbb{Z}} w^i \w$, which if $w \neq 1$ is an axis of $w$. 
\end{enumerate}

\begin{figure}[ht]
\centering
\def\svgwidth{90mm}
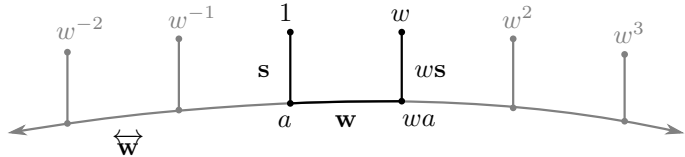
\caption[The geodesics $\a$, $\w$ and $w\a$]{The geodesics $\a$, $\w$ and $w\a$. 
The translates of $\w$ form an axis $\overleftrightarrow{\w}$.}
\label{standardfig}
\end{figure}

Notice that if $1$ lies on an axis of $w$, then $a_w=1$.  Also, our geodesics are in $X^{(1)}$, so these choices are not unique.  Henceforth, we'll omit the subscripts when there is no ambiguity.

\begin{lem}[Standard Form]
\label{standard}
Given an element $w$ in $A$ and a choice of $\w$, $a$ and $\a$ as above, 

\begin{enumerate}
\item No hyperplane meets $\displaystyle \W=\overleftrightarrow{\w} \cup \bigcup_{i \in \mathbb{Z}} w^i \a$ more than once.
\item The concatenation of $\a$, $\w$ and $w \a$ is a geodesic from $1$ to $w$.
\end{enumerate}
\end{lem}

\begin{notation}
Given a set of hyperplanes $\H$ and an edge path $\gamma$, let $\H_\gamma$ denote the set of hyperplanes in $\H$ that meet $\gamma$ an odd number of times.  This is equivalent to saying $\H_\gamma$ is the set of $h \in \H$ that separate the endpoints of $\gamma$, if it has them.  Note $\H_{\gamma'} = \H_\gamma$ for any $\gamma'$ with the same endpoints.  If $\alpha, \beta, \gamma$ is a triangle, then $\H_\gamma = \H_\alpha \Delta \H_\beta$ (the symmetric difference).
\end{notation}

\begin{proof}[Proof of Lemma \ref{standard}]

(1) Suppose some $h$ meets $\W$ more than once.  We can assume (translating by $w$) that one of those intersections lies on $\a$, since no hyperplane meets an axis twice.  We can further assume that $h$ is adjacent to $a$.  If it is separated from $a$ by some other hyperplane, then that hyperplane also intersects $\W$ more than once.  

Let $p= h \cap \a$ and let $q$ be a different point of $h\cap \W$.  By Lemma \ref{adjacent}, $[wa, wp]$ meets only hyperplanes of intersecting type with $wh$.  Thus $h$ does not meet $[wa, wp]$.  Therefore, by the Slope Lemma (\ref{slope}), the point $q$ cannot lie on $\overleftrightarrow{\w}$ beyond $wa$. Similarly $q$ cannot lie on $\overleftrightarrow{\w}$ beyond $w^{-1}a$.  This leaves the following possibilities for $q$:

\begin{itemize}
\item $h$ meets $\overleftrightarrow{\w}$ between $w^{-1}a$ and $wa$.
\item $h$ does not meet $\overleftrightarrow{\w}$ at all.  Thus $q \in w^n\a$ for some $n \neq 0$ and by the Parallel Axis Lemma (\ref{parallel}), $wh=h$.
\end{itemize}

Thus, exploiting symmetry, we can take $h$ to cross $\w$ or we can stipulate $h=wh$ (but not both, since no hyperplane meets $\overleftrightarrow{\w}$ twice).  In either case, we reach the same contradiction as follows:

Let $a'$ be the vertex adjacent to $a$ that shares the edge through $h$.   Let $\H$ be the set of all hyperplanes of $X$.  Then we have $\H_{[1, a']}=\H_\a - \{h\}$.  We also have $\H_{[a', wa']}=\{h\} \Delta \H_{\w} \Delta \{wh\}$, which in either case means that $d(a', wa')=d(a, wa)$. We conclude $a'$ lies on an axis and is closer to $1$ than $a$, violating our construction of $a$.  

(2) Since the concatenation of $\a$, $\w$ and $w \a$ crosses every hyperplane between $1$ and $w$ once, but none of them twice, it must be a geodesic segment.
\end{proof}

\begin{rem}
The use of Lemma \ref{adjacent} is a convenience here.  One can prove this lemma with only the assumption that $X$ is CAT(0), and hyperplanes of $X/A$ don't self-intersect. 
\end{rem}

\begin{lem}[Separating Lemma]
\label{separating}
Let $w \neq 1$ be an element of $A$ and suppose that a hyperplane $h$ is disjoint from one (and hence every) $X^{(1)}$ axis of $w$. Suppose that $wh \neq h$.  Then there is a hyperplane $h_1 \dis h$ that separates $h$ from $wh$. $h_1$ does not intersect $w^i h$ for any $i\in \mathbb{Z}$.  
\end{lem}

\begin{proof}
Let $[p, a]$ be an edge path from $h$ to an axis of $w$, minimal in length among all such paths for all axes of $w$.  $p$ is not a vertex, so we'll name the vertices of the edge it lies on.  Let $q$ be the one that lies in $[p,a]$.  Set the other vertex as the basepoint $1$.  Then $[1, w\cdot 1]$ decomposes into a geodesics $\a$, $\w$ and $w\a$, with $\a=[1, a] \supset [q, a]$.   By the Standard Form Lemma (\ref{standard}),
\begin{align*}
h\cap [1, w] &= \{p\} \\
wh \cap [1, w] &= \{wp\} \\
w^ih \cap [1, w] &=\emptyset \textrm{ for } i \neq 0, 1.
\end{align*} 
We conclude that $[q, wq]$ meets no $w^ih$.

Suppose that $[q, wq]$ meets no hyperplanes $h_1 \dis h$.  Then by Lemma \ref{adjacent}, $wq$ is adjacent to both $h$ and $wh$.  But $w$ preserves orientation, and only one edge of each type and direction meets the vertex $wq$.  This implies that $wh=h$, which violates our hypothesis.  We conclude that $[q,wq]$ does cross a hyperplane $h_1 \dis h$, which must therefore separate $h$ from $wh$.  For all $i \notin \{0, 1\}$, the hyperplane $w^ih \dis h_1$.  Notice though, that $h_1$ meets $[q, wq]$ and $w^ih$ doesn't.  Thus $w^ih$ can't be equal to $h_1$, so they are disjoint. 
\end{proof}

Repeated application of this lemma gives a corollary.

\begin{corollary}
\label{separating2}
As in the lemma let $w \neq 1$ be an element of $A$ and suppose that a hyperplane $h_0$ is disjoint from one (and hence every) $X^{(1)}$ axis of $w$. Suppose that $wh_0 \neq h_0$.  Then there is a hyperplane $h_N$ meeting $\overleftrightarrow{\w}$ that separates $h_0$ from $wh_0$, and does not intersect $w^i h_0$ for any $i\in \mathbb{Z}$.  
\end{corollary}
\begin{proof}
We'll argue by induction.  Given an $h_n$ that satisfies the hypotheses of Lemma \ref{separating} we produce $h_{n+1}$.  If $h_{n+1}$ crosses $\overleftrightarrow{\w}$, then we are done.  If not, then it crosses $\a$ or $w\a$, but not both.  This means that $h_{n+1}$ separates either $h_n$ or $wh_n$ from $\overleftrightarrow{\w}$.  Furthermore, $wh_n \neq h_n$, and we can apply the Lemma again.  Since $\a$ crosses finitely many hyperplanes, repeated application of the lemma will eventually produce some $h_N$ that crosses $\overleftrightarrow{\w}$.  Each hyperplane $w^ih_0$ is separated from $h_N$ by a sequence of hyperplanes $w^{i_n}h_n$.
\end{proof}

\section{Two-generator subgroups}


\subsection{Essential hyperplanes}


A two-generator subgroup of a RAAG $A$ is the image of a homomorphism $\phi:F\to A$, where $F$ is a free group on two generators.  $A$ acts on $X$, the universal cover of its Salvetti complex, so $F$ acts on $X$ too.  We'll produce the following tree $T$ containing $F$ and produce a map $\phi: T \to X$ which extends a standard orbit map on $F$.

First we assume that $F=\langle u,v \rangle$ and we choose a basepoint $1$ on some axis of $u$.  We then choose $\u$, $a_v$, $\a_v$, $\v$ as defined in the previous section (note that $a_u=1$, so we don't bother with $\a_u$), and immediately drop the subscripts. 

We define $T$ to be a trivalent tree with vertices labeled $w$ and $wa$ for all $w\in F$.  The edges of $T$ are as follows:

\begin{enumerate}
\item $w$ is adjacent to $wu$ via edge $w\u$
\item $w$ is adjacent to $wa$ via edge $w\a$
\item $wa$ is adjacent to $wva$ via edge $w\v$
\end{enumerate}

\begin{figure}[ht]
\centering
\def\svgwidth{90mm}
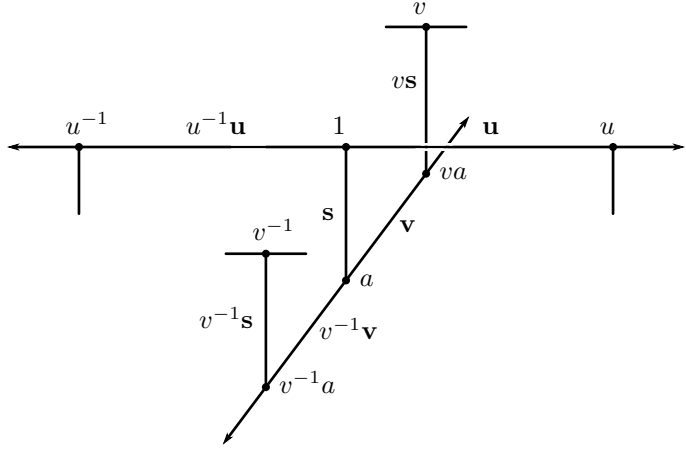
\caption{A diagram of $T$ near the vertex $1$}
\end{figure}

We extend the map $\phi$ to $T$ by mapping each vertex to its eponymous vertex in $X$ and each edge to its eponymous geodesic in some equivariant way (say, constant speed).  This has the effect that for some $q \in T$ with  $\phi(q) \in h$, we have $\phi(wq) \in h$ if and only if $wh=h$.

\begin{defn}
We say a hyperplane $h$ is \emph{essential} if $\phi^{-1}(h)$ is a single point.
\end{defn}

The advantage of essential hyperplanes to a metric argument should be clear.  A reduced word $w \in F$ can be taken as a geodesic in $T$.  If it crosses $\phi^{-1}(h)$ once, and $h$ is essential, then it will never cross $\phi^{-1}(h)$ again.  Thus our strategy is to produce sufficiently many of these to justify the main theorem.  Our tool for proving that a hyperplane is essential is the following definition and lemma.

\begin{defn}
Suppose $p \in T$.  We denote $T_p$ to be the closure of the connected component of $T-\{wp ~|~ w \in F - \{1\} \}$ containing $p$.  See Figure \ref{tpfig}.
\end{defn}

\begin{figure}[ht]
\subcaptionbox{$p \in \a$}
[.4\linewidth]
{\def\svgwidth{.4\linewidth} 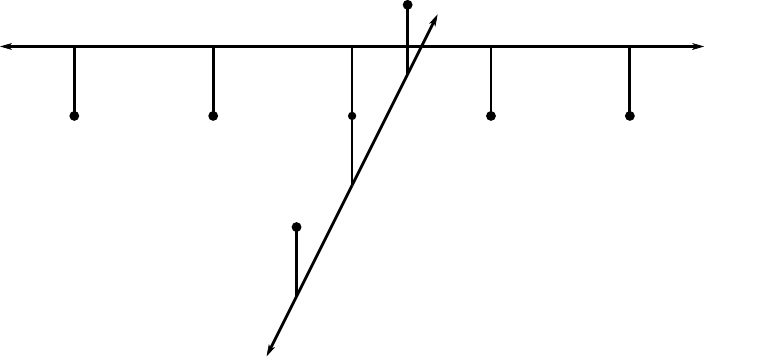}%
\subcaptionbox{$p \in \v$}
[.4\linewidth]
{\def\svgwidth{.4\linewidth} 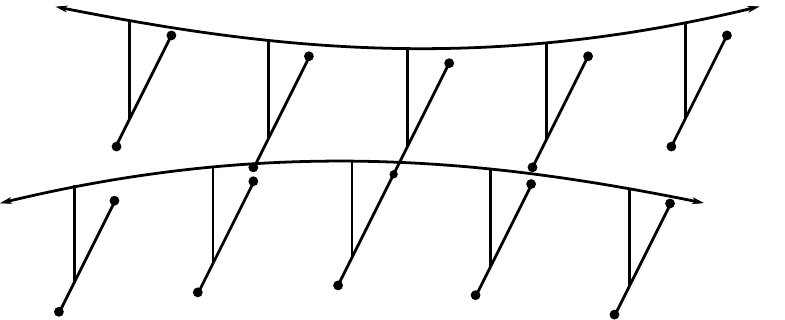}
\caption{Diagrams of $T_p$ for two possible $p$.}
\label{tpfig}
\end{figure}

\begin{lem}
\label{essentialclosure}
If $\phi^{-1}(h)\cap T_p = \{p \}$, then $h$ is essential.
\end{lem}

\begin{proof}
There are three types of edges in $T$, and within each type, the preimages of hyperplanes are identically distributed (and finite).  Thus if $h$ is not essential, we can pick a point $p'$ of $\phi^{-1}(h) \setminus \{p\}$ whose distance to $p$ in $T$ is minimal.  

Let $\gamma$ be the geodesic in $T$ from $p$ to $p'$.  Since $p' \notin T_p$ by hypothesis,  $\gamma$ contains some point $q = wp$ in the boundary of $T_p$. By assumption, $q \notin h$ so $wh \neq h$.  Since $p$ and $p'$ lie of the same side of $wh$, we conclude that $\gamma$ must cross $wh$ again at some $q'$, either between $p$ and $q$ or between $q$ and $p'$.   Translating by $w^{-1}$ we have that $w^{-1}q'$ lies in $\phi^{-1}(h)$ and is closer to $p$ than $p'$, contradicting our choice of $p'$.
\end{proof}

\subsection{The existence of essential hyperplanes}

The following proposition comprises much of the work in proving Theorem \ref{main}.  Recall the following notation:  Given a set of hyperplanes $\H$, and a geodesic $\gamma$, $\H_\gamma$ denotes the set of hyperplanes of $\H$ that intersect $\gamma$.

\begin{prop}
\label{existence}
If $\phi(F)$ is not abelian, then we can choose a basepoint $1$ and generators $u, v$ of $F$ such that for the associated $T$, there exists an essential hyperplane $h$ in $X$ meeting $[1, v]$.
\end{prop}

\begin{proof}

Let $\{ u_\textrm{orig}, v_\textrm{orig}\}$ be a generating set for $F$.  Pick a hyperplane that crosses an axis of the commutator $c_\textrm{orig}=u_\textrm{orig} v_\textrm{orig}u_\textrm{orig}^{-1}v_\textrm{orig}^{-1} $.  Let $\H$ be the set of all hyperplanes of that type.

Now pick a new generating set $u$, $v$ and a basepoint $1$ such that the following triple of numbers is minimal (lexicographically) among all possible choices of $u$ , $v$ and $1$:
\begin{align*}
( \#\H_{\u} , \#\H_{[1, v]}, \#\H_{\v} ).
\end{align*}

We claim $\H_{[1, v]}$ contains an essential hyperplane. There are two cases which, unfortunately, admit almost no overlap in their discussion.

{\bf Case 1:} Suppose $\H_{\u}$ is nonempty.  We will argue that each of $\H_{\u} , \H_\v$ contains an essential hyperplane. 

Since the hyperplanes of $\H$ are pairwise disjoint, we can order the $\H_{\u}$ by distance from $1$, and let $h_1$, $h_2$ be the middle ones (if $\#\H_{\u}$ is odd, $h_1=h_2$). We claim that one of $h_1, h_2$ does not meet the nearby $v$-paths, seen also in Figure (\ref{doesntmeetfig}): 
\begin{align*}
[1, v], [1, v^{-1}], [u, uv], [u, uv^{-1}].
\end{align*}
\begin{figure}[ht]
\subcaptionbox{The four $v$ paths avoided by one of $h_1$ or $h_2$ \label{doesntmeetfig}}
[.4\linewidth]
{\def\svgwidth{.4\linewidth} 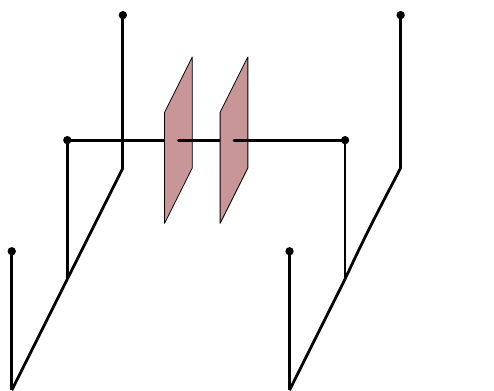}
\subcaptionbox{The hyperplanes between $1$, $u$ and $uv^{-1}$, assuming $h_1$ meets $[u, uv^{-1}]$ \label{farvpath}}
[.4\linewidth]
{\def\svgwidth{.4\linewidth} 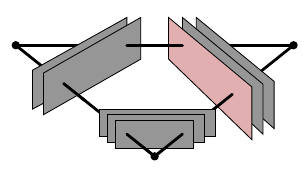}
\caption{$h_i$ and the nearby $v$-paths}
\end{figure}

Similarly if we know that $\H_\v$ is nonempty, otherwise we could switch $u$ and $v$ and take basepoint $a$ to reduce the triple.  Take $h_1'$ and $h_2'$ to be the middle hyperplanes of $\H_{\v}$.  Then these are also the middle hyperplanes of $\H_{[1,v]}$, since $\#\H_\a=\#\H_{v\a}$.  We make an analogous claim that one of $h_1'$ and $h_2'$ meets none of the nearby $u$-paths.  The following three facts establish our claims about the $h_i$ and $h_i'$ ($i \in \{1, 2\}$).

\begin{enumerate}
\item Neither $h_i$ meets either of the ``far" $v$-paths.  For instance, if $h_1$ met $[u, uv^{-1}]$, then $[1, uv^{-1}]$ would meet fewer hyperplanes than $[u,uv^{-1}]$, as seen in Figure \ref{farvpath}. Defining $v_\textrm{new} = uv^{-1}$ would give 
\begin{align*}
\#\H_{[1, v_\textrm{new}]} < \#\H_{[u, uv^{-1}]} =\#\H_{[1, v]}.
\end{align*}  
Thus $\{ u ,v_\textrm{new}\}$ would produce a smaller triple) than $\{ u, v\}$, violating our minimality assumption.  Similarly, neither $h_i'$ meets either of the far $u$-paths.

In the case that $h_1=h_2$, all four $v$-paths are ``far,'' so the claim is proven.

\item The $h_i$ do not each meet the near $v$-paths on the same side. For instance, suppose $h_1$ met $[1, v]$.  We claim that $h_2$ does not meet $[u, uv]$.  If it did, then $h_2 = uh_1$ as in Figure \ref{samesidepaths}, but no hyperplane meets $\overleftrightarrow{\u}$ twice.  Similarly, since $h_1'$ and $h_2'$ meet the axis of $v$, we have $h_1'$ and $h_2'$ do not both meet the near $u$-paths on the same side.

\item The $h_i$ do not each meet the near $v$-paths on the opposite sides.  For instance, suppose $h_1$ met $[1, v]$ and $h_2$ met $[u, uv^{-1}]$.  

Choose a point $e \in \u$ between $h_1$ and $h_2$.  Notice $\H_{[1, e]}$ consists of $h_1$ and all the hyperplanes of $\H$ that separate it from $1$, while $\H_{[e, u]}$ consists of $h_2$ and all the hyperplanes of $\H$ that separate it from $u$.

By part (1), $h_1$ does not meet the far $v$-path $[u, uv^{-1}]$, so it meets $[1, vu^{-1}]$.  Also by part (1) no hyperplane of $\H_{[1, v]}$ meets both $[1, u]$ and $[v, vu^{-1}]$ (one of these $u$-paths would be ``far").  Since $h_1 \in \H_{[1, v]}$, it meets $[u, uv^{-1}]$ (see Figure \ref{oppsidepaths}).

Thus if we set $v_\textrm{new}=uv^{-1}$ we have that $h_1$ meets all of $[1, u], [1,v_\textrm{new}]$ and  $[1, v^{-1}_\textrm{new}]$.   
So $\H_{[1, e]}$ is a subset of $\H_{\u}$,  $\H_{[1, v_\textrm{new}]}$, and  $\H_{[1, v^{-1}_\textrm{new}]}$.  Translating the last inclusion by $v_\textrm{new}$ gives $\H_{[v_\textrm{new}, v_\textrm{new}e]} \subset \H_{[1, v_\textrm{new}]}$.  Let's see what this implies.

\begin{align*}
\H_{[e, ue]}&= \H_{[1, e]} \Delta \H_\u   \Delta \H_{[u, ue]} =  ( \H_\u - \H_{[1, e]} ) \cup \Delta \H_{[u, ue]} \\
\textrm{so }\# \H_{[e, ue]} &= \#\H_\u.  \\
\\
\H_{[e, v_\textrm{new}e]} &=\H_{[1, e]} \Delta \H_{[1, v_\textrm{new}]} \Delta \H_{[v_\textrm{new}, v_\textrm{new}e]} \\ 
&=\H_{[1, v_\textrm{new}]} - (\H_{[1, e]} \Delta \H_{[v_\textrm{new}, v_\textrm{new}e]} ) \subsetneq  \H_{[1, v_\textrm{new}]} \\
\textrm{so }\# \H_{[e, v_\textrm{new}e]} &< \# \H_{[1, v_\textrm{new}].} 
\end{align*}

Out hypothesis on $h_2$ gives $\H_{[e, u]}= \H_u \cap  \H_{[u,uv]}$, so  $\H_{[1, v_\textrm{new}]} =  \H_u \Delta \H_{[u,uv]} = \H_{[1, e]} \sqcup (\H_{[u,uv]}-\H_{[e, u]})$.  So $\# \H_{[1, v_\textrm{new}]}= \# \H_{[1, v]}$.

This means that replacing $v$ by $v_\textrm{new}$ and shifting the basepoint to $e$ would strictly reduce the triple, violating our assumptions about its minimality under $u$, $v$ and $1$.  The $v$ case is similar.
\end{enumerate}

\begin{figure}[ht]
\subcaptionbox{Meeting the same side paths means $h_1$ and $h_2$ don't cross $\overleftrightarrow{u}$. \label{samesidepaths}}
[.4\linewidth]
{\def\svgwidth{.4\linewidth} 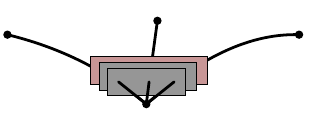}
\subcaptionbox{The points $e$ and $v_\textrm{new}=uv^{-1}$ when $h_1$ and $h_2$ meet opposite side paths \label{oppsidepaths}}
[.4\linewidth]
{\def\svgwidth{.4\linewidth} 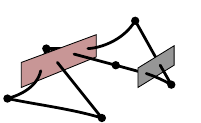}
\caption{Illustration of facts (2) and (3)}
\end{figure}

This proves our claim about $h_1$ and $h_2$ in $\H_\u$.  Call one that meets none of the nearby $v$-paths $h_u$.  Similarly call one of $h_1', h_2'$ that meets none of the nearby $u$-paths $h_v$.   Furthermore, $h_v$ is separated from each edge $[v^n, v^nu^{\pm 1}]_{n \notin \{0, 1\}}$ by one of its own $v$-translates.  

We'll use Lemma \ref{essentialclosure} to show that $h_u$ is essential.  Consider $p=h_u \cap \u$.  Then to see that $T_p \cap h_u = \{p\}$ we check:

\begin{enumerate}
\item $\overleftrightarrow{\u} \cap h_u = \{p\}$
\item $h_u$ doesn't meet the four nearby $v$-paths: $[1, v], [1, v^{-1}], [u, uv], [u, uv^{-1}]$ by our choice of $h_u$.
\item $h_u$ does not meet any point that lies across any of the hyperplanes: $h_v, uh_v, v^{-1}h_v$, or  $uv^{-1}h_v$ (see Figure \ref{case1tp}), since none of these cross $h_u$.
\end{enumerate}

\begin{figure}[ht]
\centering
\def\svgwidth{90mm}
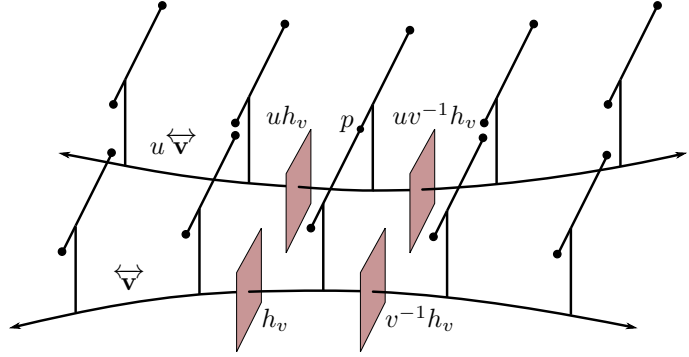
\caption[Verifying $h_u$ is essential]{Verifying $h_u$ is essential.  Notice $h_u$ doesn't cross these 4 translates of $h_v$.} 
\label{case1tp}
\end{figure}

This verifies that $T_p \cap h_u = \{p\}$  and by Lemma \ref{essentialclosure}, $h_u$ is essential.  By identical reasoning, so is $h_v$.  As it lies in $\H_{[1,v]}$, $h_v$ is the hyperplane we sought.

{\bf Case 2:} Suppose $\H_{\u}$ is empty.  Recall that $\H$ is the set of all hyperplanes of some type that crosses the axis of $c_\textrm{orig}$, the commutator of the original generators of $F$.  However, one can check that $c=uvu^{-1}v^{-1}$ is equal to $g c_\textrm{orig}^{\pm 1} g^{-1}$ for some $g \in A$ (the reader can check this on his or her favorite generating set of $Aut(F)$).  Thus $\H$ also contains a hyperplane that crosses any axis of $c$.  Consider a standard decomposition of $[1, c]$ into $\a_c$, $\axisc$ and $c \a_c$.  Since $\H$ is $c$-invariant, there is some $h \in \H$ that meets $c^{-1}\axisc$.  Our goal is to show that this $h$ is essential.  

Consider the geodesic from $1$ to $c^{-1}$ in $T$.  It maps to a piecewise geodesic path $\xi$ in $X$, which is a concatenation of two translates of $[1,v]$ and two translates of $\u$.  Since $\H_u$ is empty, $h$ meets $\xi$ at either $[1, v]$ or $[vu, vuv^{-1}]$, but not both (see Figure \ref{xi}).  Assume, without loss of generality, that it meets $[1, v]$ at $p$.   

\begin{figure}[ht]
\centering
\def\svgwidth{90mm}
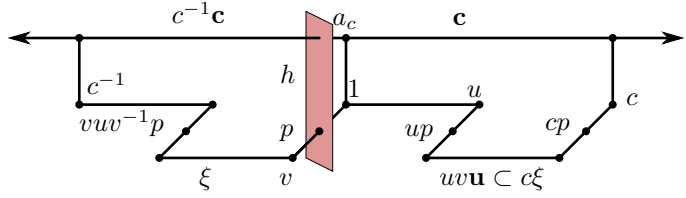
\caption[The intersection of $h$ with $\xi$]{The intersection of $\phi^{-1}(h)$ with $\xi$.  If $h \in \H_{[1,v]}$ intersects $c^{-1}\axisc$ and $\H_\u$ is empty, then $\phi^{-1}(h)$ meets $\xi$ exactly once and doesn't meet $c^{-1}\xi$ at all.}
\label{xi}
\end{figure}

We produce two facts about the stabilizer of $h$.  Since $vuv^{-1}p \notin h$, we know $vuv^{-1}h \neq h$.  We also claim that $uh \neq h$.  If we suppose otherwise, then $h$ meets $c\xi$ at $up$.  Since $h$ cannot cross $[1, c]$ or any translate of $\u$, it must meet $c\xi$ a second time, and that meeting must be some $q  \in [uvu^{-1}, c]$.  Since $h$ doesn't meet $uv\overleftrightarrow{\u}$, the Parallel Axis Lemma says that $q=cp$.  But if $ch=h$, then $h$ cannot meet $\overleftrightarrow{\axisc}$ at all, violating our choice of $h$.  Thus $uh \neq h$ as claimed.  We'll use the fact that neither $uh$ nor $vuv^{-1}h$ is equal to $h$ in the argument that follows.
 
We'll now argue that $h$ is essential using Lemma \ref{essentialclosure}.  To do this we will consider the intersection of $h$ with the following 5 sets (illustrated in Figure \ref{case2}).  The preimages of these sets cover $T_p$ (where $p$ now refers to $\phi^{-1}(h) \cap [1,v] \in T$) regardless of whether $p$ lies on $\a$, $\v$ or $v\a$.  Computing these intersections will show that $\phi^{-1}(h) \cap T_p = \{p\}$.  

\begin{enumerate}
\item $\displaystyle \V = \overleftrightarrow{\v} \cup \bigcup_{n \in \mathbb{Z}} v^n\a$.  $h$ meets $\V$ only at $p$, by Lemma \ref{standard}.
\item $\overleftrightarrow{\u} \cup v\overleftrightarrow{\u}$.  The intersection of $h$ with these axes is empty, because no $F$-translate of $h$ meets $\u$.
\item $\displaystyle \bigcup_{n \neq 0} [u^n, u^nv]$.  Since $h$ doesn't meet $\overleftrightarrow{\u}$, the Parallel Axis Lemma (\ref{parallel}) states that if $h$ meets one of these edges then $uh =h$. We've shown that this isn't the case, so $h$ does not intersect this set.
\item $\displaystyle \bigcup_{n \neq 0} [vu^nv^{-1}, vu^n]$.  Note as in the previous step that $h$ doesn't meet $v\overleftrightarrow{\u}$ and $vuv^{-1}h \neq h$. Thus $h$ does not intersect this set.
\item Finally, in the case $p \in \v$,  the region $T_p$ overlaps the paths $[u^nv^{-1}, u^n]$ and $[vu^n, vu^nv]$.  We'll use the following {\bf $v$-straightening} argument to show that $h$ doesn't meet $[u^nv^{-1}, u^n]$, noting that a similar one exists for $[vu^n, vu^nv]$:

Suppose $h$ meets $[u^nv^{-1}, u^n]$.  Since $\H_\u$ is empty, $[1, v]$ meets the same hyperplanes as $[1, vu^{-n}]$.  So if we set  $v_\textrm{new}=vu^{-n}$, then $\H_{[1, v_\textrm{new}]}=\H_{[1,v]}$.   We may then produce a $T_\textrm{new}$ with $\v_\textrm{new}$ and $\a_\textrm{new}$.  But $h$ and every hyperplane of $\H$ between $1$ and $h$ (including those of $\H_\a$), meet both $[1, v_\textrm{new}]$ and $[1, v^{-1}_\textrm{new}]$.  Thus $h \cup \H_\a \subset \H_{\a_\textrm{new}}$, which means $\# \H_{\v_\textrm{new}} <  \#\H_\v$, violating the minimality of the triple.
\end{enumerate}

\begin{figure}[ht]
\centering
\def\svgwidth{90mm}
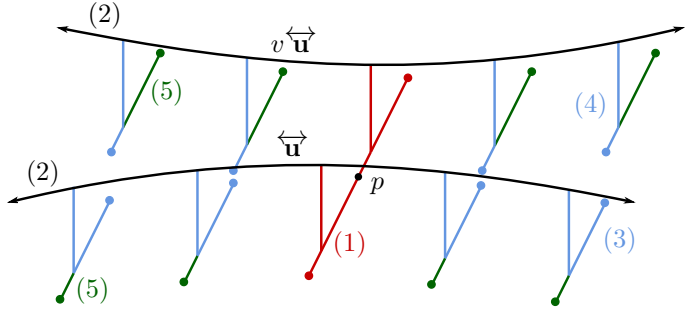
\caption[$T_p$ is divided into 5 regions]{$T_p$ is divided into 5 regions to verify that $h \cap T_p = \{p\}$.  This is the $p \in \v$ case.}
\label{case2}
\end{figure}

This establishes that $\phi^{-1}(h) \cap T_p = \{p\}$.  By Lemma \ref{essentialclosure}, $h$ is essential.
\end{proof}

The following corollary may be of interest, but we won't use it here.

\begin{corollary}
For $F \to A$ a homomorphism and any hyperplane $h$, there is a choice of generators $u, v \in F$ and a basepoint $1$ such that either
\begin{enumerate}
\item $h$ does not intersect $T$.
\item There exists $w\in F$ such that $wh$ is essential at some $p \in [1,v]$.
\item There exists $w\in F$ such that $wh$ meets $[1, v]$ and either $uh=h$ or $vuv^{-1}h=h$ (possibly both).
\end{enumerate}
\end{corollary}

\begin{proof}
Apply the methods of the previous proposition to $\H=Fh$.  Assuming the conditions of case 1 leads to essential hyperplanes $h_u \in \H_\u$ and $h_v \in \H_{\v}$.  Since $h_u = wh_v$ for some $w \in F$, we have that $h_v$ meets $w^{-1}\u$.  This contradicts the claim that it is essential.  In case 2 we showed that every hyperplane of $\H_{[1, v]}$ is essential unless $uh=h$ or $vuv^{-1}h=h$.  So either $\H_{[1, v]}$ is empty (along with $\H_\u$ by assumption) or some $w$-translate of $h$ satisfies (2) or (3).
\end{proof}

\subsection{The main theorem}

Using Proposition \ref{existence}, we are ready to prove the main theorem: that $\phi$ is a quasi-isometric embedding.

\begin{proof}[Proof of Theorem \ref{main}]
For a word $w$ of length $n$, we know that $$|\phi(w) | \leq n \max \{|\phi(u) |,|\phi(v) | \}.$$  It remains to find a lower bound that is linear in $n$.  We'll produce a pair of essential hyperplanes, meeting \u and \v.  We'll consider their orbits $Fh_u$ and $Fh_v$, and count how many of these hyperplanes cross $[1,w]$.  Note that in case 1 of Proposition \ref{existence} we already produced an adequate pair of essential hyperplanes, but the methods here only rely on the assumption that some essential $h$ crosses $[1,v]$.

Take the choice of $u, v$ and a basepoint from Proposition \ref{existence}.  Consider all possible generators $v_\textrm{new}=u^{n_1}vu^{n_2}$, and the $T_\textrm{new}$ generated by $\langle u , v_\textrm{new} \rangle$, but keeping the same basepoint.  Let $h$ be the hyperplane of $X$ that lies closest to $\overleftrightarrow{u}$ among all essential (with respect to $T_\textrm{new}$) hyperplanes meeting $[1, v_\textrm{new}]$.  Such hyperplanes exist by Proposition \ref{existence}, and a closest one exists because $\overleftrightarrow{\u}$ is the $u$ orbit of the compact set $\u$.  Let $p = \phi^{-1}(h)$.

The Separating Lemma (\ref{separating}) part (1) states that there is a hyperplane $h' \dis h$ between $h$ and $uh$.  Since any $F$-translate $h$ meets $T$ only once, we know $h' \notin Fh$ so it is disjoint from all $F$-translates of $h$.  As a result, $h' \cap T$ is contained in $T_p$.  In fact, it is contained in the component of $T_p -\{p\}$ that contains $uh$.  

We claim $h'$ crosses $\overleftrightarrow{\u}$.  Suppose it does not.  We will argue that $h'$ satisfies the conditions that lead us to choose $h$ but lies closer to $\overleftrightarrow{\u}$.  $h'$ meets $[p, up]$, so $h'$ meets either $[1,p]$ or $[u, up]$ but not both.  This means that $uh' \neq h'$.  Without loss of generality, suppose it meets $[1,p]$.  The contrapositive of the Parallel Axis Lemma (\ref{parallel}) therefore applies, meaning $h'$ does not meet any $[u^n, u^np]$ for $n \neq 0$.  The only other possible point of $\phi^{-1}(h')$ is on some $[u^nv^{-1}, u^n]$.  If $h'$ doesn't meet any such edge, then it is essential and closer to $\overleftrightarrow{\u}$ than $h$.  This contradicts our choice of $h$.  If $h'$ does meet this edge, then set $v_\textrm{new}=vu^{-n}$, and one may check that $T_\textrm{new} \cap \phi^{-1}(h')$ is a single point on $\a_\textrm{new}$.  Thus $h'$ would be essential for this $v_{\textrm{new}}$ and closer to $\overleftrightarrow{\u}$ than $h$.  This also contradicts our choice of $h$.

So $h'$ meets some $u^m \u$.  Now we claim that for some $v_\textrm{new}$, both $h$ and $h'$ are essential on $T_\textrm{new}$.  That $h$ is essential is immediate.  If $v_\textrm{new}=u^{n_1}vu^{n_2}$, then $[1, v_\textrm{new}]$ meets $u^{n_1}h$ and no other $F$-translates.  Thus $u^{n_1}h$ is essential on $[1, v_\textrm{new}]$.

As noted above, $h' \cap T$ is limited to one component of $T_p - \{p\}$.  If $h'$ meets any $u$-translate of $[1,p]$, then let $u^{k_1}[1,p]$ be the one farthest from $u^m\u$.  Repeated applications of the Slope Lemma show that only finitely many translates of $[1,p]$ meet $h'$.   They also show that $h'$ meets every translate of $[1,p]$ between $u^{k_1}[1, p]$ and $u^m\u$.  We define $u^{k_2}[1,v^{-1}p]$ similarly to produce the farthest translate meeting $h'$, and apply the Slope Lemma similarly. Then we can choose $v_\textrm{new}=u^{\sigma(m, k_1)}vu^{-\sigma(m, k_2)}$ where 

\begin{align*}
\sigma(m, k) = \begin{cases} k-m-1 & \textrm{if } k \leq m \\
k-m & \textrm{if } k > m \\
0 & \textrm{if } k \textrm{ is not defined}
\end{cases}
\end{align*}

One can check that $h'$ is essential, meeting $T_\textrm{new}$ only once, on $u^m\u$.

\begin{figure}[ht]
\centering
\def\svgwidth{90mm}
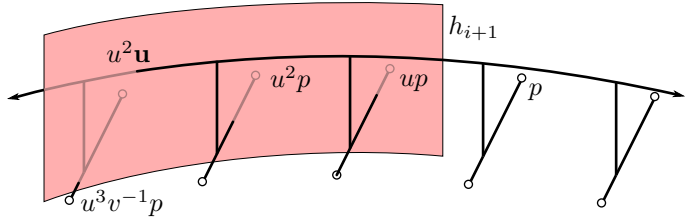
\caption[An example of a non-essential $h'$ meeting $u^2\u$]{An example of a non-essential $h'$ meeting $u^2\u$.  The set $h' \cap T$ is limited to one component of $T_p-\{p\}$.   In this case $k_1=1$ and $k_2=3$.}
\label{uhyperplane}
\end{figure}

Now let $h_u=u^{-m}h'$.  As noted above $u^{\sigma(m, k_1)}h$ is essential.  It lies on either $\v_\textrm{new}$,  $\a_\textrm{new}$ or $v_\textrm{new}\a_\textrm{new}$.  In the first case, we'll call it $h_v$.  

In the second two cases, Corollary \ref{separating2} produces a $h_v$ meeting $\v_\textrm{new}$ that doesn't intersect $u^{\sigma(m, k_1)}h$ or any $u^{\sigma(m, k_1)}v_\textrm{new}^nh$.   This restriction, along with the Standard Form Lemma (\ref{standard}), means that $h_v$ is essential.  

Either way we now have essential hyperplanes $h_u$ and $h_v$.  The group $F$ acts freely on the orbits $Fh_u$ and $Fh_v$, which consist entirely of essential hyperplanes.  Any geodesic in $T_\textrm{new}$ which crosses the preimage $\phi^{-1}(h)$ of an essential hyperplane $h$ will never cross that preimage again. Since we have one such hyperplane in $Fh_u \cup Fh_v$ for each translate of $\u$ and $\v_\textrm{new}$, a reduced word $w$ of length $n$ in $\{u, v_\textrm{new}\}$ will cross exactly $n$ hyperplanes of $Fh_u \cup Fh_v$.

We have been flexible with our choice of generators.  Given a different set, say $\{ u_\textrm{orig}, v_\textrm{orig}\}$, we might have to compose with some automorphism of $F$.  But automorphisms are quasi-isometries, and quasi-isometric embeddings are closed under composition.
\end{proof}

\begin{rem}
Our procedure is nearly algorithmic.  While some steps assume that a certain quantity is minimized, we also showed explicitly how to reduce it, should the properties we required be absent.  Computationally one needs to be able to check whether some $wh$ is equal to $h$ and whether some $[p,q]$ crosses $h$.  One also must be able to produce the hyperplanes of the Separating Lemma.  Neither of these tools is more difficult than placing an element of $A_\Gamma$ in a standard form.  Given those computational tools, the initial description of $\phi$, and the steps of the proof, one can produce $1$, $u$, $v$, and the essential $h_u$, $h_v$.  
\end{rem}

The non-abelian case of Baudisch's Theorem is recovered by the following corollary.  The results listed in the introduction follow quickly.

\begin{corollary}
If $\phi(F)$ is not abelian, then $\phi$ is an injection.
\end{corollary}

\begin{proof}
By the previous theorem, given any choice of generators (and hence word metric), $\phi$ is a quasi isometric embedding.  There is a $\lambda \geq 1$ and $\epsilon\geq 0 $ such that for all $w\in F$,   

$$\frac{1}{\lambda}|w| -\epsilon \leq |\phi(w)|.$$ 

Suppose $w\in F$ is not the identity. It translates along its axis in $T$ by some distance $d$.  For sufficiently large $n$, we have $|w^n|\geq nd >\lambda\epsilon$.  But then we have 

$$0 = \frac{1}{\lambda}\lambda\epsilon -\epsilon < \frac{1}{\lambda}|w^n| -\epsilon  \leq |\phi(w^n)|.$$

Since $\phi(w^n)$ is not the identity, neither is $\phi(w)$.
\end{proof}

\begin{proof}[Proof of Theorem \ref{2gen}]
Let $H$ be a two-generator subgroup of $A_\Gamma$.  RAAGs  are torsion free \cite[2.3]{Bau81}, so if $H$ is abelian, it is free abelian. If $H$ is not abelian, then apply the previous lemma.
\end{proof}

\begin{proof}[Proof of Corollary \ref{2genqie}]
As above, suppose first $H$ is abelian.  Free abelian subgroups of isometries of $X$ are quasi-isometrically embedded in $X$ (and hence in $A_\Gamma$) by the Flat Torus Theorem (presented in \cite[II.7.1 and II.7.17]{bh99}).  If $H$ is not abelian then Theorem \ref{main} applies.
\end{proof}

\section{Counterexamples in other settings}

\subsection{A three-generator subgroup}

The result of Baudisch cannot be extended to subgroups of more generators.  He produces a counterexample himself \cite[6.2]{Bau81}, which we repeat here.  Let $\Gamma$ be the following graph:

\begin{center}
\def\svgwidth{25mm}
\begingroup%
  \makeatletter%
  \providecommand\color[2][]{%
    \errmessage{(Inkscape) Color is used for the text in Inkscape, but the package 'color.sty' is not loaded}%
    \renewcommand\color[2][]{}%
  }%
  \providecommand\transparent[1]{%
    \errmessage{(Inkscape) Transparency is used (non-zero) for the text in Inkscape, but the package 'transparent.sty' is not loaded}%
    \renewcommand\transparent[1]{}%
  }%
  \providecommand\rotatebox[2]{#2}%
  \ifx\svgwidth\undefined%
    \setlength{\unitlength}{42.7828125bp}%
    \ifx\svgscale\undefined%
      \relax%
    \else%
      \setlength{\unitlength}{\unitlength * \real{\svgscale}}%
    \fi%
  \else%
    \setlength{\unitlength}{\svgwidth}%
  \fi%
  \global\let\svgwidth\undefined%
  \global\let\svgscale\undefined%
  \makeatother%
  \begin{picture}(1,0.65328147)%
    \put(0,0){\includegraphics[width=\unitlength]{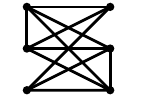}}%
    \put(-0.00613564,0.56951901){\color[rgb]{0,0,0}\makebox(0,0)[lb]{\smash{$a$}}}%
    \put(-0.00613564,0.28903254){\color[rgb]{0,0,0}\makebox(0,0)[lb]{\smash{$b$}}}%
    \put(-0.00613564,0.00854607){\color[rgb]{0,0,0}\makebox(0,0)[lb]{\smash{$c$}}}%
    \put(0.83532376,0.56951901){\color[rgb]{0,0,0}\makebox(0,0)[lb]{\smash{$x$}}}%
    \put(0.83532376,0.28903254){\color[rgb]{0,0,0}\makebox(0,0)[lb]{\smash{$y$}}}%
    \put(0.83532376,0.00854607){\color[rgb]{0,0,0}\makebox(0,0)[lb]{\smash{$z$}}}%
  \end{picture}%
\endgroup%

\end{center}

Write $F_3= \langle u, v, w\rangle$ and let $\phi: F_3 \to A_\Gamma$ be defined by 
\begin{align*}
\phi(u)=ax \qquad \phi(v)=by \qquad \phi(w)=cz 
\end{align*}

Notice that no two of $\phi(u), \phi(v), \phi(w)$ commute, yet Baudisch shows that the image of $\phi$ is not a free group, meaning that $\phi$ is not an injection.  He goes farther and shows that the image of $\phi$ not even a RAAG.   

\subsection{An obstruction to embedding in a RAAG}

Consider the group $\pi_1(M)$, where $M$ is the figure-8 knot complement.  $\pi_1(M)$ can be realized as a semi-direct product $F_2 \rtimes \mathbb{Z}$ with the following presentation.
\begin{align*}
\pi_1(M)=\langle a, b, t ~|~ tat^{-1}=ab, ~ tbt^{-1}=bab \rangle
\end{align*}

Now let $F_2=\langle a, b\rangle$ and $i: F_2 \to \pi_1(M)$ be the inclusion.  The subgroup $F_2$ is free.

Consider the element $t^nat^{-n}$, which has length at most $2n+1$ in $\pi_1(M)$.  We can equate this to a word in $a, b$.  We use the relations to cancel a power of $t$ and $t^{-1}$ and repeat $n$ times.  Each cancellation at least doubles the number of $a$'s and $b$'s.  No $a$'s or $b$'s ever cancel, since we only produce positive powers of $a$ and $b$.  The result is that $t^nat^{-n}$ is equal to a word of length at least $2^n$ in $a$ and $b$.  Thus there is an infinite sequence of elements $w_n \in F_2$ such that $|i(w_n)| \leq 2 \log_2 |w_n| +1$.    Thus $i$ is not a quasi-isometric embedding.  We can thus apply the following corollary to $\pi_1(M)$.

\begin{corollary}\label{noembed}
Let $G$ be any finitely-generated group with a free two-generator subgroup $F_2=\langle a, b\rangle$ that is not quasi-isometrically embedded.  No RAAG has a subgroup isomorphic to $G$.
\end{corollary}

\begin{proof}
Consider the inclusion $i: F_2 \to G$.  As noted at the beginning of the proof of Theorem \ref{main}, for all $w \in F_2$, we have an upper bound $|i(w)| \leq |w| \max \{|i(a)|,|i(b) | \}$.  Thus in order for $i$ to fail to be a quasi-isometric embedding, $F_2$ must contain an infinite sequence of elements $w_n$ such that $|i(w_n)|$ is sublinear with respect to $|w_n|$.  

Let $\{g_i\}$ be a generating set for $G$.  If there were an injective homomorphism $\alpha: G \to A_\Gamma$, then $\alpha \circ i$ would be an injection.  For all $w_n$ we would have $|\alpha(w_n))| \leq |i(w_n)| \max\{|\alpha(g_i)|\}$, which would still be sublinear with respect to $|w_n|$.  Thus the map $\alpha \circ i$ violates Theorem \ref{main}.
\end{proof}

Note however that, while $\pi_1(M)$ is not isomorphic to a subgroup of any RAAG, it is a hyperbolic free by cyclic HNN extension.  The theorem of Hsu and Wise \cite{HW10} (reproduced as Theorem \ref{hsuwise} here) applies.  Thus $\pi_1(M)$ has a finite index subgroup $H$ that is isomorphic to a subgroup of some RAAG.  It would be interesting to identify such an $H$.  Here is a first step.

\begin{prop}
If $H$ is a finite index normal subgroup of $\pi_1(M)$ with abelian quotient, then $H$ is not isomorphic to a subgroup of any RAAG.
\end{prop}

\begin{proof}
Let $H$ be an index $m$ normal subgroup of $\pi_1(M)$.  If $\pi_1(M)/H$ is abelian, then the commutator subgroup, $[\pi_1(M), \pi_1(M)]$, is a subgroup of $H$.  So $a^{-1}tat^{-1} = b \in H$.  Also, $tbt^{-1}b^{-1}=ba \in H$, so $a\in H$.  Therefore, $F_2 \subset H$.  Finally, $t^m \in H$.  Thus we have a sequence of $w_n\in F_2$ such that $i(w_n)=t^{mn}at^{-mn}$ and $|i(w_n)| \leq 2 \log_2 |w_n| +1$.  Thus $H$ is not isomorphic to a subgroup of any RAAG by Corollary \ref{noembed}.
\end{proof}

Hsu and Wise's construction in \cite{HW10} gives a cube complex $C$ with an action by $\pi_1(M)$.  The proposition above shows that the action of $H$ is not special.  Equivalently, no abelian cover of $C/\pi_i(M)$ is special.

\begin{ques}
What is the smallest index subgroup of $\pi_1(M)$ that embeds in a RAAG?   Does this number change for different hyperbolic groups of the form $F_2 \rtimes \mathbb{Z}$?
\end{ques}

\begin{ques}
Can we find groups acting on CAT(0) cube complexes such that no subgroup of a given index acts has a special action?  Can we find these in the class of $3$-manifold groups?
\end{ques}

\subsection{A badly embedded free subgroup}

Baudisch's three-generator subgroup shows that the most naive generalization of our main theorem fails to hold when we increase the number of generators of $F$.  However, even assuming that $F_n$ injects into $A_\Gamma$ is insufficient to guarantee a quasi-isometric embedding.  

As noted before, if $M$ is the figure-8 knot complement, then $\pi_1(M)$ does not embed in any RAAG, but some finite index subgroup $H < \pi_1(M)$ does.  The group $H \cap F_2$ is a finite index subgroup of $F_2$ and thus finitely generated \cite{Sch27}, \cite[3.9]{LS77}.  Call it $F_n$.  The inclusions of $H$ into $\pi_1(M)$ and $F_n$ into $F_2$ are quasi-isometries (this is a standard exercise).  

\begin{align*}
\xymatrix@R=.5pc{ 
F_n \ar[dd]_{i_F} \ar[r]^i & H \ar[dd]^{i_{\pi_1(M)}} \ar[r]^\alpha_{\textrm{q.i.}}&  A_\Gamma\\
\\
F_2 \ar[r]_i & \pi_1(M) 
}
\end{align*}

It is straightforward to show that when a quasi-isometric embedding is composed with a map that is not a quasi-isometric embedding, the result is not a quasi-isometric embedding.  We conclude that $\alpha \circ i$ is an injective homomorphism from free group to a RAAG that is not a quasi-isometric embedding.

\begin{ques}
What is the minimum rank $n$ such that some RAAG has a free subgroup of rank $n$ whose inclusion is not a quasi-isometric embedding?
\end{ques}

\bibliographystyle{alpha}
\bibliography{CAT0}
\end{document}